\documentclass{amsart}[12pt]

\newtheorem{theorem}{Theorem}
\newtheorem{lemma}[theorem]{Lemma}
\newtheorem{proposition}[theorem]{Proposition}
\newtheorem{corollary}[theorem]{Corollary}

\theoremstyle{definition}

\usepackage{amscd,amssymb}

\begin{document}

\title[Weak identities in the algebra of symmetric matrices]
{Weak identities in the algebra\\
of symmetric matrices of order two}
\author[V.S. Drensky]{Vesselin S. Drensky}

\address{Institute of Mathematics and Informatics,
Bulgarian Academy of Sciences,
Acad. G. Bonchev Str., Block 8, 1113 Sofia, Bulgaria}
\email{drensky@math.bas.bg}

\thanks{Published as Vesselin S. Drensky,
Weak identities in the algebra of symmetric matrices of order two (Russian), Pliska Stud. Math. Bulgar. 8 (1986), 77-84.
Zbl 0668.16009, MR0866647.}
\date{}

\maketitle

\begin{abstract}
We describe the weak polynomial identities of the Jordan algebra of symmetric $2\times 2$ matrices over a field of characteristic zero.
The corresponding weak verbal ideal is generated by the standard identity of degree four and the metabelian identity.
\end{abstract}

Let $K_2$ be the algebra of $2\times 2$ matrices over a field $K$ and let $H(K_2)$ be the Jordan algebra of the symmetric matrices in $K_2$.
A.M. Slinko \cite[Problem 2.96]{2} stated the problem to find the basis of the weak identities in the pair $(K_2,H(K_2))$
in the case of a field of characteristic zero. A partial answer was given in \cite{10} where the description was given of the module structure
of the relatively free pair corresponding to the weak T-ideal $T(K_2,H(K_2))$. The main purpose of the present paper is to give
the complete answer to the problem of A.M. Slinko:

\noindent{\bf Theorem.}
{\it Let $K$ be a field of characteristic $0$. Then the basis of the weak identities of the pair $(K_2,H(K_2))$ consists of the standard identity
\begin{equation}\label{eq 1}
S_4(x_1,x_2,x_3,x_4)=0
\end{equation}
and the metabelian identity
\begin{equation}\label{eq 2}
[[x_1,x_2],[x_3,x_4]]=0.
\end{equation}
}

\section{Preliminaries}\label{Section 1}
In the sequel $K$ will be a fixed field of characteristic 0. All associative and Jordan algebras will be unitary and over $K$.
The existence of the unit does not decrease the generality of the considerations because both algebras $K_2$ and $H(K_2)$ are unitary.
All necessary information on identities of Jordan algebras can be found in \cite{5}. The notation is similar to that in \cite{10} and \cite{9}.

Let us denote by $A_m$ be the free associative algebra $A(x_1,\ldots,x_m)$ with free generators $x_1,\ldots,x_m$ and let $A=A_{\infty}$.
Additionally,
\[
x_1\circ x_2=x_1x_2+x_2x_1,\quad
x_1\circ\cdots\circ x_{n-1}\circ x_n=(x_1\circ\cdots\circ x_{n-1})\circ x_n,
\]
\[
[x_1,x_2]=x_1x_2-x_2x_1,\quad
[x_1,\ldots,x_{n-1},x_n]=[[x_1,\ldots,x_{n-1}],x_n].
\]
Let $B_m^{(n)}$ be the vector subspace of $A_m$ spanned by the products of commutators $[x_{i_1},\ldots]\cdots[\ldots,x_{i_n}]$,
$\displaystyle B_m=\sum_{n\geq 0}B_m^{(n)}$, and let $P_n$ be the set of multilinear polynomials of degree $n$ in $A_n$.
Then $\Gamma_n=P_n\cap B_n^{(n)}$ is the subset of the proper polynomials in $P_n$.
The vector spaces $P_n$ and $A_m$ have, respectively, the structure of left $\text{Sym}(n)$- and $GL(m,K)$-modules (see e.g. \cite[\S 1]{3}),
where $\text{Sym}(n)$ is the symmetric group acting on the set $\{1,\ldots,n\}$ and $GL(m,K)$ is the general linear group.
The subspaces $\Gamma_n$ and $B_m^{(n)}$ are, respectively, submodules of $P_n$ and $A_m$.
The irreducible $\text{Sym}(n)$- and $GL(m,K)$-modules are described by Young diagrams and partitions
$\lambda=(\lambda_1,\ldots,\lambda_r)$ of $n$. We shall denote the corresponding modules by $M(\lambda)$ and $N_m(\lambda)$.

The algebra $A_m$ is the universal enveloping algebra of the free Lie algebra $L_m$. Using the Poincar\'e-Birkhoff-Witt theorem
it is easy to show that if $f_{ks}(x_1,\ldots,x_k)$, $s=1,\ldots,\gamma_k$, is a basis of the vector space $\Gamma_k$, then $P_n$ has a basis
\[
x_{i_1}\cdots x_{i_{n-k}}f_{ks}(x_{j_1},\ldots,x_{j_k}),\quad i_1<\cdots< i_{n-k},\quad j_1<\cdots< j_k,
\]
\[
\{i_1,\ldots,i_{n-k},j_1,\ldots,j_k\}=\{1,\ldots,n\},\quad s=1,\ldots,\gamma_k,\quad k=0,1,\ldots,n.
\]
Similarly, if $g_{ks}(x_1,\ldots,x_m)$, $s=1,\ldots,\beta_k$, is a basis of $B_m^{(k)}$, then $A_m$ has a basis
\[
x_1^{\alpha_1}\cdots x_m^{\alpha_m}g_{ks}(x_1,\ldots,x_m),\quad \alpha_i\geq 0, i=1,\ldots,m, \quad s=1,\ldots,\beta_k,k=0,1,2,\ldots.
\]

The algebra $A_m$ has an involution $\ast$ defined by the equality
\[
(x_{i_1}\cdots x_{i_n})^{\ast}=x_{i_n}\cdots x_{i_1}.
\]
The Jordan algebra of the symmetric elements $H(A_m,\ast)$ contains the free special Jordan algebra $SJ_m$.
By the theorem of P.M. Cohn \cite[p. 76 of the Russian original]{5} $H(A_m,\ast)=SJ_m$ for $m\leq 3$.

In the sequel we shall use that
\[
[x_1,\ldots,x_n]^{\ast}=(-1)^{n-1}[x_1,\ldots,x_n]
\]
and the commutators of odd length are Jordan elements, i.e. belong to $SJ=SJ_{\infty}$.

Let $G$ be a special Jordan algebra and let $R$ be its associative enveloping algebra. By analogy with \cite[Definitions 1--3]{7}
the polynomial $f(x_1,\ldots,x_n)$ in $A$ is a weak identity for the pair $(R,G)$ if $f(g_1,\ldots,g_n)=0$ for all $g_1,\ldots,g_n\in G$.
The set $T=T(R,G)$ of all weak identities of the pair $(R,G)$ is a weak T-ideal (or a weak verbal ideal) in $A$.
The polynomials $\{f_i(x_1,\ldots,x_{n_i})\}$ generate $T$ as a weak T-ideal (i.e. are a basis of $T$), if $T$ is generated as an ordinary ideal
by the set $\{f_i(u_1,\ldots,u_{n_i})\mid u_j\in SJ\}$.
If follows from the unitarity of the algebras $R$ and $G$ and from the above comments on the bases of $P_n$ and $A_m$
that the basis of the identities of $T$ can be chosen in $\bigcup(\Gamma_n\cap T)$, $n\geq 2$.
Similarly, all identities in $m$ variables in $T$ follow from $\bigcup(B_m^{(n)}\cap T)$, $n\geq 2$.
By \cite[Lemma 2.3]{9} the $\text{Sym}(n)$-module $\Gamma_n/(\Gamma_n\cap T)$ and the $GL(m,K)$-module $B_m^{(n)}/B_m^{(n)}\cap T)$
have the same structure: If
\[
\Gamma_n/(\Gamma_n\cap T)\cong \sum k_{\lambda}M(\lambda),
\]
then
\[
B_m^{(n)}/B_m^{(n)}\cap T)\cong \sum k_{\lambda}N_m(\lambda).
\]
(In all the paper the sums of modules are direct.)

The proof of the next lemma repeats the proof of \cite[Lemma 1]{10}:

\begin{lemma}\label{Lemma 1}
Let $G$ be a Lie algebra with an ordered basis $g_1<g_2<\cdots$ and let $U(G)$ be its universal enveloping algebra. Then $U(G)$ has a basis
\[
g_{i_1}\circ g_{i_2}\circ \cdots\circ g_{i_r},\quad i_1\geq i_2\geq\cdots\geq i_r.
\]
\end{lemma}

\begin{corollary}\label{Corollary 2}
As a vector space the algebra $A_3$ is spanned by the elements $u$ and $u[v,w]$, where $u,v,w\in SJ_3$.
\end{corollary}

\begin{proof}
We choose an ordered basis of leftnormed commutators in the free Lie algebra $L_3$, $u_1<u_2<\cdots<v_1<v_2<\cdots$,
where $u_i$ (respectively $v_j$) are commutators of even (odd) length (i.e. $u_i^{\ast}=-u_i$, $v_j^{\ast}=v_j$).
By Lemma \ref{Lemma 1} the algebra $A_3$ has a basis consisting of the polynomials
\[
t=v_{j_1}\circ \cdots\circ v_{j_n}\circ u_{i_1}\circ\cdots\circ u_{i_m},\quad j_1\geq \cdots\geq j_n,\quad i_1\geq \cdots\geq i_m.
\]
If $m$ is even, then $t^{\ast}=t$ and by the theorem of P.M. Cohn $t\in SJ_3$. If $m$ is odd, then $t=t_1\circ u_{i_m}$, $t_1\in SJ_3$.
Therefore $t$ is a Jordan evaluation of $x_1$ or $x_1\circ[x_2,x_3]$ (instead of $x_1$ we may have 1). But
\[
x_1\circ[x_2,x_3]=-[x_1,[x_2,x_3]]+2x_1[x_2,x_3]
\]
and $[x_1,[x_2,x_3]]\in SJ_3$ which completes the proof.
\end{proof}

\section{Weak Capelli identities}\label{Section 2}
Recall that the $k$-th weak Capelli identities are polynomials in $P_n$, $n\geq k$, which are alternating in the variables $x_1,\ldots,x_k$.
The Capelli identities are linear combinations of multilinear polynomials of the form
\[
\sum(-1)^{\sigma}u_1x_{\sigma(1)}u_2x_{\sigma(2)}\cdots u_kx_{\sigma(k)}u_{k+1},\quad \sigma\in\text{Sym}(k),
\]
and $u_1,\ldots,u_{k+1}$ are monomials.

\begin{lemma}\label{Lemma 3}
Let $(R,G)$ be a pair and let $T$ be the corresponding weak T-ideal, let $F_m=A_m/(A_m\cap T)$
and let $\widetilde{B}_m$ and $\widetilde{\Gamma}_n$ be the images of $B_m$ and $\Gamma_n$ under the canonical homomorphism
$A_m\to F_m$. Let all polynomials in $\widetilde{\Gamma}_n$ with $k$ alternating variables, $n\geq k$, are equal to $0$. Then

{\rm (i)} $\widetilde{B}_m\cong \sum N_m(\lambda_1,\ldots,\lambda_{k-1})$;

{\rm (ii)} $\widetilde{F}_m\cong \sum N_m(\mu_1,\ldots,\mu_k)$;

{\rm (iii)} The pair $(R,G)$ satisfies all $(k+1)$-th weak Capelli identities.
\end{lemma}

\begin{proof}
(i) By \cite[Theorem 2]{11} the condition that all polynomials with $k$ alternating variables disappear means that
the irreducible components of the $\text{Sym}(n)$-module $\widetilde{\Gamma}_n$ correspond to Young diagrams with not more than $k-1$ rows, i.e.
$\widetilde{\Gamma}_n\cong \sum M(\lambda_1,\ldots,\lambda_{k-1})$.
In virtue of the correspondence between the module structures of $\widetilde{\Gamma}_n$ and $\widetilde{B}_m^{(n)}$ we obtain that
$\widetilde{B}_m\cong \sum N_m(\lambda_1,\ldots,\lambda_{k-1})$.

(ii) It follows from \cite[Theorem 2.6]{9} and \cite[Proposition 2]{10} that
\[
F_m\cong K[x_1,\ldots,x_m]\otimes_K\widetilde{B}_m,
\]
where $K[x_1,\ldots,x_m]$ is the ordinary polynomial algebra. But
\[
K[x_1,\ldots,x_m]\cong\sum N_m(n),\quad n\geq 0.
\]
Using the rule for the tensor product of $GL(m,K)$-modules \cite[Chapter 8]{1} we obtain that
\[
N_m(n)\otimes N_m(\lambda_1,\ldots,\lambda_{k-1})\cong\sum N_m(\lambda_1+n_1,\ldots,\lambda_{k-1}+n_{k-1},n_k),
\]
where $0\leq n_1$, $0\leq n_i\leq\lambda_{i-1}-\lambda_i$, $2\leq i\leq k-1$, $0\leq n_k\leq \lambda_{k-1}$,
$n_1+\cdots+n_k=n$. Hence the irreducible submodules of $F_m$ have Young diagrams with not more than $k$ rows.

(iii) The statement follows from \cite[Theorem 2]{11} and the correspondence between the $\text{Sym}(n)$- and $GL(m,K)$-module
structure of $P_n(P_n\cap T)$ and $F_m$ \cite[\S 1]{3}.
\end{proof}

Till the end of the paper we shall denote by $T$ the weak T-ideal generated by the identities (\ref{eq 1}) and (\ref{eq 2}).
It follows from \cite[Proposition 2.1]{4} that
\begin{equation}\label{eq 3}
\Gamma_4\cong M(3,1)+M(2^2)+M(2,1^2)+M(1^4).
\end{equation}
From here it is easy to see that the identities (\ref{eq 1}) and (\ref{eq 2}) generate $M(2,1^2)$ and $M(1^4)$
in this decomposition and are equivalent to the identity
\begin{equation}\label{eq 4}
\sum(-1)^{\sigma}[x_{\sigma(1)},x_{\sigma(2)}][x_{\sigma(3)},x_4]=0,\quad \sigma\in\text{Sym}(4).
\end{equation}
The pair $(K_2,H(K_2))$ satisfies the weak identities (\ref{eq 1}) and (\ref{eq 2}). It is well known for (\ref{eq 1})
and (\ref{eq 2}) can be checked directly because the commutator of two symmetric matrices is skew-symmetric and is proportional to $e_{12}-e_{21}$.
Hence
\begin{equation}\label{eq 5}
T\subseteq T(K_2,H(K_2)).
\end{equation}
The theorem will be established if we show that in (\ref{eq 5}) there is an equality.

In the sequel we shall work in $F=A/T$.

\begin{proposition}\label{Proposition 4}
All polynomials in three alternating variables in $\widetilde{\Gamma}_n$ are equal to zero.
\end{proposition}

\begin{proof}
We shall proceed by induction on $n$.
The base of the induction $n=4$ holds because in the decomposition (\ref{eq 3}) the modules $M(2,1^2)$ and $M(1^4)$ belong to $T$.
Let $f(x_1,\ldots,x_n)\in\widetilde{\Gamma}_n$ is alternating in $x_1,x_2,x_3$. First we shall consider the case $n=5$. By \cite[Theorem 2.3]{3}
the $\text{Sym}(5)$-module of the Lie elements in $\Gamma_5$ decomposes as
\[
P_5(L)\cong M(4,1)+M(3,2)+M(3,1^2)+M(2^2,1)+M(2,1^3).
\]
By \cite[Remark 2.8]{3} the sum of the latter three submodules is generated by the identity $[[x_1,x_2],[x_3,x_4],x_5]=0$
which is a weak consequence of (\ref{eq 2}). Hence, if the considered polynomial $f$ is a Lie element, then it vanishes in $\widetilde{\Gamma}_5$.
By \cite[Remark 1.2 and Proposition 2.4]{4} we can work in $\Gamma_5$ modulo $P_5(L)$ and
\[
\Gamma_5/P_5(L)\cong M(3,2)+M(3,1^2)+M(2^2,1)+M(2,1^3).
\]
We substitute in (\ref{eq 2}) $x_1$ by the Jordan element $x_1^2$ and obtain the consequence
\[
0=[[x_1,x_2]\circ x_1,[x_3,x_4]]=[[x_1,x_2],[x_3,x_4]]\circ x_1+[x_1,x_2]\circ[x_1,[x_3,x_4]],
\]
i.e. $F$ satisfies
\begin{equation}\label{eq 6}
[x_1,x_2]\circ[x_3,x_4,x_1]=0.
\end{equation}
In the proof of \cite[Lemma 3.2]{4} we established that modulo $P_5(L)$ the identities from the submodules $M(3,1^2)$, $M(2^2,1)$ and $M(2,1^3)$
of $\Gamma_5$ follow from (\ref{eq 6}). In this way we complete the proof for $n=5$. Later we shall need also that the identity
\begin{equation}\label{eq 7}
[x_2,x_1,x_1]\circ[x_2,x_1]=0
\end{equation}
is also a consequence of (\ref{eq 6}).

Now we shall consider the general case. By \cite[page 154 of the Russian original]{6} every element of $\widetilde{\Gamma}_n$ can be written as a linear combination
of products of canonical commutators $[x_{i_1},x_{i_2},\ldots,x_{i_k}]$, $i_1>i_2<\cdots<i_k$. Besides
\[
\sum(-1)^{\sigma}[x_{\sigma(1)},x_{\sigma(2)},x_{\sigma(3)}]=0,
\]
\[
\sum(-1)^{\sigma}[x_i,x_{\sigma(1)},x_{\sigma(2)}]=[x_i,x_1,x_2]-[x_i,x_2,x_1]
\]
\[
=-[x_1,x_2,x_i]=-\frac{1}{2}\sum(-1)^{\sigma}[x_{\sigma(1)},x_{\sigma(2)},x_i],
\]
\[
\sum(-1)^{\sigma}[x_{i_1},x_{i_2},x_{\sigma(1)},x_{\sigma(2)}]=\frac{1}{2}\sum(-1)^{\sigma}[[x_{i_1},x_{i_2}],[x_{\sigma(1)},x_{\sigma(2)}]].
\]
Hence we may assume that the alternating variables are in the most left position in two or three commutators
and $f(x_1,\ldots,x_n)$ is a linear combination of polynomials of the following kinds:
\begin{equation}\label{eq 8}
\begin{array}{c}
\sum(-1)^{\sigma}u_1[x_{\sigma(1)},x_{\sigma(2)},x_i,\ldots]u_2[x_{\sigma(3)},x_j,\ldots]u_3,\\
\\
\sum(-1)^{\sigma}u_1[x_{\sigma(1)},x_i,\ldots]u_2[x_{\sigma(2)},x_{\sigma(3)},x_j,\ldots]u_3,\\
\\
\sum(-1)^{\sigma}u_1[x_{\sigma(1)},x_i,\ldots]u_2[x_{\sigma(2)},x_j,\ldots]u_3[x_{\sigma(3)},x_k,\ldots]u_4.
\end{array}
\end{equation}
Here the summation is on $\sigma\in\text{Sym}(3)$ and $u_1,u_2,u_3,u_4$ are products of commutators.
We shall consider the first and the third cases. The second case is similar.
We express the commutators $u_1,u_2,u_3,u_4$ as a linear combination of monomials and in
$[x_{\sigma(1)},x_{\sigma(2)},x_i,\ldots]$, $[x_{\sigma(3)},x_j,\ldots]$,
$[x_{\sigma(1)},x_i,\ldots]$, $[x_{\sigma(2)},x_j,\ldots]$, $[x_{\sigma(3)},x_k,\ldots]$ we leave only the inner commutators of length 2.
In this way we write (\ref{eq 8}) as a linear combination of
\begin{equation}\label{eq 9}
\begin{array}{c}
\sum(-1)^{\sigma}v_1[x_{\sigma(1)},x_{\sigma(2)}]v_2[x_{\sigma(3)},x_j]v_3,\\
\\
\sum(-1)^{\sigma}w_1[x_{\sigma(1)},x_i]w_2[x_{\sigma(2)},x_j]w_3[x_{\sigma(3)},x_k]w_4,\\
\end{array}
\end{equation}
where $v_1,v_2,v_3,w_1,w_2,w_3,w_4$ are monomials. Without loss of generality we may assume that $v_1=v_3=w_1=w_4=1$.
The degree of the monomials $v_2,w_2$ and $w_3$ is lower than $n$. Hence, by the inductive assumption, modulo the fourth Cappeli identity
(Lemma \ref{Lemma 3} (iii)), $v_2,w_2$ and $w_3$ are equivalent to identities in three variables.
By Corollary \ref{Corollary 2}, in (\ref{eq 9}) we may assume that $v_2$ and $w_2$ are replaced by $1, y_1,[y_1,y_2],[y_1,y_2]y_3$,
and $w_3$ is replaced by $1,y_4,[y_4,y_5],y_4[y_5,y_6]$. It follows from (\ref{eq 2}) that $[x_1,x_2][x_3,x_4]=[x_3,x_4][x_1,x_2]$
and we can move the commutator $[y_1,y_2]$ to the first position, e.g.
\[
\sum(-1)^{\sigma}[x_{\sigma(1)},x_{\sigma(2)}][y_1,y_2][x_{\sigma(3)},x_j]=[y_1,y_2]\sum(-1)^{\sigma}[x_{\sigma(1)},x_{\sigma(2)}][x_{\sigma(3)},x_j].
\]
Using arguments for symmetry instead of (\ref{eq 9}) it is sufficient to consider the cases
\begin{equation}\label{eq 10}
z=\sum(-1)^{\sigma}[x_{\sigma(1)},x_{\sigma(2)}]y[x_{\sigma(3)},x_4],
\end{equation}
\begin{equation}\label{eq 11}
z_1=\sum(-1)^{\sigma}[x_{\sigma(1)},x_4][x_{\sigma(2)},x_5][x_{\sigma(3)},x_6],
\end{equation}
\begin{equation}\label{eq 12}
z_2=\sum(-1)^{\sigma}[x_{\sigma(1)},x_4]y_1[x_{\sigma(2)},x_5][x_{\sigma(3)},x_6],
\end{equation}
\begin{equation}\label{eq 13}
z_3=\sum(-1)^{\sigma}[x_{\sigma(1)},x_4]y_1[x_{\sigma(2)},x_5]y_2[x_{\sigma(3)},x_6].
\end{equation}
Since the statement of the proposition is true for polynomials of degree 4 and 5, we obtain that in $F$
\[
z=y\sum(-1)^{\sigma}[x_{\sigma(1)},x_{\sigma(2)}][x_{\sigma(3)},x_4]+\sum(-1)^{\sigma}[x_{\sigma(1)},x_{\sigma(2)},y][x_{\sigma(3)},x_4]=0.
\]
We shall write (\ref{eq 10}) in a more detailed form:
\[
2([x_1,x_2]y[x_3,x_4]+[x_2,x_3]y[x_1,x_4]-[x_1,x_3]y[x_2,x_4])=0,
\]
\begin{equation}\label{eq 14}
\sum(-1)^{\sigma}[x_{\sigma(1)},x_3]y[x_{\sigma(2)},x_4]=\frac{1}{2}\sum(-1)^{\sigma}[x_{\sigma(1)},x_{\sigma(2)}]y[x_3,x_4].
\end{equation}
Similarly, for $y=1$ (or from (\ref{eq 4})) we obtain
\begin{equation}\label{eq 15}
\sum(-1)^{\sigma}[x_{\sigma(1)},x_3][x_{\sigma(2)},x_4]=\frac{1}{2}\sum(-1)^{\sigma}[x_{\sigma(1)},x_{\sigma(2)}][x_3,x_4].
\end{equation}
We apply the identities (\ref{eq 14}) and (\ref{eq 15}) to (\ref{eq 11}), (\ref{eq 12}) and (\ref{eq 13}) and obtain that they follow from
the identities (\ref{eq 4}) and (\ref{eq 10}), e.g.
\[
z_2=\frac{1}{2}\left(\sum(-1)^{\sigma}[x_{\sigma(1)},x_4]y_1[x_{\sigma(2)},x_{\sigma(3)}]\right)[x_5,x_6]=0.
\]
The proof of the proposition is completed.
\end{proof}

\begin{corollary}\label{Corollary 5}
The weak identity
\begin{equation}\label{eq 16}
\begin{array}{c}
[x_1,x_2,x_4,\ldots]\cdots[x_3,x_5,\ldots]\\
\\
=[x_1,x_3,x_4,\ldots]\cdots[x_2,x_5,\ldots]-[x_2,x_3,x_4,\ldots]\cdots[x_1,x_5,\ldots].\\
\end{array}
\end{equation}
holds in the algebra $F$.
\end{corollary}

\begin{proof}
The corollary follows immediately from Proposition \ref{Proposition 4}
because the variables $x_1,x_2,x_3$ alternate in the identity (\ref{eq 16}).
\end{proof}

\section{Weak identities in two variables}\label{Section 3}

In the proof of \cite[Proposition 3]{10} it was established that
\begin{equation}\label{eq 17}
B_2/(B_2\cap T(K_2,H(K_2)))\cong\sum N_2(p+q,p),\quad p>0,q\geq 0.
\end{equation}
In virtue of Proposition \ref{Proposition 4} and the embedding (\ref{eq 5}) the main theorem will be proved
if we show that the Hilbert series of the modules $B_2/(B_2\cap T(K_2,H(K_2)))$ and $\widetilde{B}_2$ coincide.
Hence till the end of the paper it is sufficient to work in $\widetilde{B}_2$.

\begin{proposition}\label{Proposition 6}
$\widetilde{B}_2^{(n)}\cong B_2^{(n)}/(B_2^{(n)}\cap T(K_2,H(K_2)))$, $n\leq 6$.
\end{proposition}

\begin{proof}
It follows from the decomposition into a sum of irreducible submodules of $B_2^{(2)},B_2^{(3)}$ and $B_2^{(4)}$
and from (\ref{eq 17}) that $B_2^{(2)},B_2^{(3)}$ and $B_2^{(4)}$ intersect trivially with $T$.
The identity (\ref{eq 7}) generates a $GL(2,K)$-module isomorphic to $N_2(3,2)$.
Since $B_2^{(5)}\cong N_2(4,1)+2N_2(3,2)$, we obtain that (\ref{eq 7}) is the only identity in two variables in $B_2^{(5)}$.
It follows from \cite[Propositions 2.5 and 2.6]{4} that
\[
B_2^{(6)}\cong N_2(5,1)+3N_2(4,2)+2N_2(3^2)
\]
and the corresponding submodules are generated by the polynomials

\noindent $N_2(5,1)$: $w_1=y(\text{ad}x)^5$;

\noindent $N_2(4,2)$: $w_2=[[y,x,x,x],[y,x]]$, $w_{16}=[y,x][y,x,x,x]$, $w_{25}=[y,x,x]^2$;

\noindent $N_2(3^2)$: $w_9=2[[y,x,x],[y,x,y]]$, $w_{13}=[y,x]^3$.

\noindent The commutator $[x_1,x_2,x_3]$ is a Jordan element and as a consequence of (\ref{eq 2}) we obtain the weak identity
$[[[x_1,x_2,x_3],x_4],[x_5,x_6]]=0$. Hence the commutators of even length commute in $F$:
\begin{equation}\label{eq 18}
[x_1,\ldots,x_{2k}][y_1,\ldots,y_{2l}]=[y_1,\ldots,y_{2l}][x_1,\ldots,x_{2k}].
\end{equation}
In particular
\begin{equation}\label{eq 19}
[[y,x,x,x],[y,x]]=0.
\end{equation}
It follows from the identities (\ref{eq 7}) and (\ref{eq 18}) that
\[
0=[[y,x,x]\circ[y,x],x]=[y,x,x,x]\circ [y,x]+[y,x,x]\circ[y,x,x],
\]
\begin{equation}\label{eq 20}
[y,x,x]^2=-[y,x,x,x][y,x].
\end{equation}
We linearize (\ref{eq 7}) in $y$:
\[
[y,x,x]\circ[z,x]+[z,x,x]\circ[y,x]=0
\]
and replace $z$ by $y^2$:
\[
0=[y,x,x]\circ[y^2,x]+[y^2,x,x]\circ[y,x]
\]
\[
=[y,x,x]\circ(y\circ[y,x])+([y,x,x]\circ y+2[y,x]^2)\circ [y,x].
\]
Working modulo (\ref{eq 18}) we obtain
\[
g(x,y)=[y,x,x,y][y,x]+[y,x,y][y,x,x]+2[y,x]^3=0.
\]
In the latter equation we change the places of $x$ and $y$ and subtract
\[
g(x,y)-g(y,x)=[[y,x,y],[y,x,x]]+4[y,x]^3=0;
\]
\begin{equation}\label{eq 21}
[[y,x,y],[y,x,x]]=-4[y,x]^3.
\end{equation}
It follows from (\ref{eq 19}), (\ref{eq 20}) and (\ref{eq 21}) that
$B_2^{(6)}\cap T\cong 2N_2(4,2)+N_2(3^2)$,
i.e.
\[
\widetilde{B}_2^{(6)}\cong B_2^{(6)}/(B_2^{(6)}\cap T(K_2,H(K_2))).
\]
\end{proof}

\begin{proposition}\label{Proposition 7}
The vector space $\widetilde{B}_2$ is spanned by
\[
([y,x](\text{\rm ad}x)^k(\text{\rm ad}y)^l)[y,x]^{q-1},\quad k,l\geq 0,\quad q\geq 1.
\]
\end{proposition}

\begin{proof}
Let us consider the $\text{Sym}(5)$-submodule $M$ of $\Gamma_5$ generated by the polynomial $[x_1,x_2,x_3]\circ[x_4,x_5]$.
By \cite[Proposition 2.4]{4} $M\cong M(3,2)+M(3,1^2)+M(2^2,1)+M(2,1^3)$. It follows from Section \ref{Section 2} that
$M$ is equal to 0 in $\widetilde{\Gamma}_5$ and in $\widetilde{\Gamma}_5$
\[
[x_1,x_2][x_3,x_4,x_5]=-[x_3,x_4,x_5][x_1,x_2].
\]
As in (\ref{eq 18}) we obtain
\begin{equation}\label{eq 22}
[x_1,\ldots,x_{2k}][y_1,\ldots,y_{2l+1}]=-[y_1,\ldots,y_{2l+1}][x_1,\ldots,x_{2k}].
\end{equation}
Similarly, let $M_1$ be the $\text{Sym}(6)$-submodule of $\Gamma_6$ generated by the polynomials
$[x_1,x_2,x_3][x_4,x_5,x_6]$ and $[x_1,x_2][x_3,x_4,x_5,x_6]$. Then
\[
M_1\cong \sum M(k_1,k_2)+\sum M(l_1,l_2,\ldots,l_m),\quad m>2,
\]
is the decomposition of $M_1$ into a sum of irreducible components.
The second summand of $M_1$ is equal to zero in $\widetilde{\Gamma}_6$ in virtue of Proposition \ref{Proposition 4}.
It follows from (\ref{eq 19}), (\ref{eq 20}) and (\ref{eq 21}) that the components $M(k_1,k_2)$ are expressed as linear combinations of
$[x_{i_1},x_{i_2},x_{i_3},x_{i_4}][x_{i_5},x_{i_6}]$ and $[x_{i_1},x_{i_2}][x_{i_3},x_{i_4}][x_{i_5},x_{i_6}]$. Hence in $\widetilde{\Gamma}_6$
\[
[x_1,x_2,x_3][x_4,x_5,x_6]=\sum\alpha_i[x_{i_1},x_{i_2},x_{i_3},x_{i_4}][x_{i_5},x_{i_6}]
\]
\[
+\sum\beta_i[x_{i_1},x_{i_2}][x_{i_3},x_{i_4}][x_{i_5},x_{i_6}]
\]
for suitable $\alpha_i,\beta_i\in K$. By analogy with (\ref{eq 18}) and (\ref{eq 22}) we derive that
\begin{equation}\label{eq 23}
[x_1,\ldots,x_{2k+1}][x_{2k+2},\ldots,x_{2l}]=\sum \gamma[x_{i_1},\ldots,x_{i_{2p}}]\cdots[x_{j_1},\ldots,x_{j_{2q}}].
\end{equation}
(All commutators in the right hand side are of even length.)

By Lemma \cite[Lemma 1.5]{8} the vector space $B_2$ is spanned by
\begin{equation}\label{eq 24}
([y,x](\text{ad}x)^{k_1}(\text{ad}y)^{l_1})\cdots ([y,x](\text{ad}x)^{k_m}(\text{ad}y)^{l_m}),\quad k_i,l_i\geq 0.
\end{equation}
Using the identities (\ref{eq 22}) and (\ref{eq 23}) we may assume that $\widetilde{B}_2$ is spanned only by the elements from (\ref{eq 24})
with even $k_m+l_m$ when $m>1$. The commutators of odd length are in $SJ_2$ and from the identity (\ref{eq 16}) we obtain that
\[
[y,x,t_1,\ldots][y,x,z_1,x_2,\ldots]=[y,x,t_1,\ldots][[y,x,z_1],z_2,\ldots]
\]
\[
=[[y,x,z_1],x,t_1,\ldots][y,z_2,\ldots]-[[y,x,z_1],y,t_1,\ldots][x,z_2,\ldots],
\]
i.e. the elements in $\widetilde{B}_2$ are linear combinations of those elements in (\ref{eq 24}) with $k_m=l_m=0$ when $m>0$.
The proof is completed by easy induction on the degree of the elements (\ref{eq 24}).
\end{proof}

\begin{proposition}\label{Proposion 8}
The Hilbert series of $B_2/(B_2\cap T(K_2,H(K_2)))$ and $\widetilde{B}_2$ coincide.
\end{proposition}

\begin{proof}
It follows from \cite[Theorem 2.2]{9} and \cite[Proposition 2]{10} that
\[
H_1(t)=H(B_2/(B_2\cap T(K_2,H(K_2))),t)=1+t^2(1-t)^{-2}(1-t^2)^{-1}.
\]
On the other hand, it follows from Proposition \ref{Proposition 7} that the coefficients of the Hilbert series 
$H_2(t)=H(\widetilde{B}_2,t)$ are bounded from above by the coefficients of the Hilbert series $H_3(t)$
of the vector subspace of $A(x,y)$ with basis 1 and $([y,x](\text{ad}x)^k(\text{ad}y)^l)[y,x]^{q-1}$, $k,l\geq 0$, $q\geq 1$.
For the series $H_3(t)$ we have
\[
H_3(t)=1+t^2\sum t^{k+l}(t^2)^{q-1}=1+t^2(1-t)^{-2}(1-t)^{-2}=H_1(t).
\]
Since $H_1(t)\leq H_2(t)$ by (\ref{eq 5}) and $H_2(t)\leq H_3(t)$ we derive that $H_1(t)=H_2(t)$.
This completes the proof of the proposition and hence also of the main theorem.
\end{proof}

\end{document}